\documentclass[12pt]{amsart} 

\title[Vanishing of characteristic classes for handlebody groups]{Vanishing of universal characteristic classes for
handlebody groups and boundary bundles} 

\author{Jeffrey Giansiracusa}
\email{j.h.giansiracusa@gmail.com}
\address{Department of Mathematical Sciences, University of Bath\\
Claverton Down \\
Bath, BA2 7AY \\ United Kingdom} 

\author{Ulrike Tillmann}
\email{tillmann@maths.ox.ac.uk}
 \address{Mathematical Institute, Oxford University\\
 24-29 St. Giles' \\
Oxford, OX1 3LB \\ United Kingdom} 
\date{9 February 2011}

\usepackage{mathptmx}
\usepackage{amssymb}
\usepackage{mathrsfs}  %some extra fonts

\usepackage[cmtip,arrow]{xy}
\usepackage{pb-diagram,pb-xy}
\usepackage[vmargin=2.8cm, hmargin=3cm,dvips]{geometry}
\numberwithin{equation}{section}

\newtheorem{theorem}{Theorem}[section] % numbered like the section
\newtheorem{lemma}[theorem]{Lemma} % numbered like the theorems
\newtheorem{proposition}[theorem]{Proposition}

\newtheorem{thmA}{Theorem}

\newtheorem{corA}[thmA]{Corollary}
\theoremstyle{remark} % styled differently... not italicized
\newtheorem{remark}[theorem]{Remark}

\newcommand{\Diff}{\mathrm{Diff}}

\newcommand{\Th}{\mathrm{Th}}
\newcommand{\bTh}{\mathbf{Th}}

\newcommand{\R}{\mathbb{R}}

\newcommand{\Z}{\mathbb{Z}}
\newcommand{\Q}{\mathbb{Q}}

\newcommand{\MT}{\mathbf{MTSO}}
\newcommand{\QQ}{Q}

\begin{document}
\begin{abstract}
  Using certain Thom spectra appearing in the study of cobordism
  categories, we show that the odd half of the Miller-Morita-Mumford
  classes on the mappping class group of a surface with negative Euler
  characteristic vanish in integral cohomology when restricted to the
  handlebody subgroup.  This is a special case of a more general
  theorem valid in all dimensions: universal characteristic classes made from
  monomials in the Pontrjagin classes (and even powers of the Euler
  class) vanish when pulled back from $B\Diff(\partial W)$ to
  $B\Diff(W)$.
\end{abstract}
\maketitle

\section{Introduction}

Let $\Sigma _g$ denote a closed oriented surface of genus $g$. Its 
mapping class group $\Gamma _g:= \pi _0 \Diff (\Sigma _g)$ is   
the group of connected
components of its  group of orientation preserving diffeomorphisms
$\Diff ( \Sigma _g)$.
Miller, Morita, and Mumford \cite{Miller, Morita, Mumford} defined
characteristic classes, known as the MMM
classes, $\kappa_i \in H^{2i} (\Gamma _g; \Z)$.  
%The classes are constructed by raising the Euler
%class of the fibrewise tangent bundle to the power $i+1$ and then
%integrating along the fibres.  
%Universally, $\kappa_i$ lives in
%$H^{2i}(\Gamma_g);\Z)$, where $\Gamma _g = \pi_0 \Diff (\Sigma _g)$
%and $\Diff(\Sigma_g)$ is the group of
%orientation-preserving diffeomorphisms of a genus $g$ surface $\Sigma_g$.  
%When
%$g\geq 2$ the components of $\Diff(\Sigma_g)$ are contractible
%\cite{Earle-Eells}, and so the MMM classes may be considered to live
%in the cohomology of the mapping class group $\Gamma_g := \pi_0
%\Diff(\Sigma_g)$.  
By the proof of the Mumford conjecture
\cite{Madsen-Weiss} these classes freely
generate the rational cohomology ring in degrees increasing with $g$:
\[
\lim _{g \to \infty} H^* (\Gamma _g; \Q ) 
\simeq \Q [ \kappa_1, \kappa _2, \dots ].
\]

The mapping class group of a surface has various interesting
subgroups, and it is a natural question to ask how the MMM-classes
restrict to these subgroups.  Here we will be interested in the
handlebody subgroup $H_g$. To define it, fix a handlebody $W$ with
boundary $
\partial W = \Sigma _g$. $H_g$ contains those mapping classes of $\Sigma _g$ 
that can be extended across the interior of $W$. 

\begin{thmA}\label{handlebody-thm}
  For $g\geq 2$, the odd MMM-classes $\kappa_{2i+1} \in
  H^{4i+2}(\Gamma_g;\Z)$ vanish when restricted to the handlebody
  subgroup $H_g \subset \Gamma _g$.
\end{thmA}

\begin{remark}
It is well-known that the analogue of Theorem A holds rationally for the 
Torelli group $I_g := \mathrm{ker} (\Gamma_g
\to \mathrm{Aut}(H_1(\Sigma_g;\Z))$. This can be proved by index theory, 
see \cite{Morita, Mumford}.
It remains a significant open problem whether the even kappa classes restrict 
non-trivially to the Torelli group.

Motivated by these questions, Sakasai \cite{Sakasai} has recently
proved a result closely related to Theorem A by rather different
methods. He shows that in a stable range the odd kappa classes
rationally vanish when restricted to the Lagrangian mapping class
subgroup $L_g := H_g I_g$. As our result holds without restriction to
the stable range and integrally, the question arises whether the same
holds also for $I_g$ and $L_g$.
\end{remark}

\begin{remark}
Recently (and after the completion of this work), Hatcher has
announced an analogue of the Madsen-Weiss theorem \cite{Madsen-Weiss}
for the handlebody mapping class group. The proof is an
adaptation of the  Galatius's proof
\cite{Galatius} of the analogue of the Madsen-Weiss theorem
for automorphism groups of free groups.  
Hatcher determines the cohomology of $H_g$
in the stable range, 
which
by
\cite{Hatcher-Wahl} is $(g-4)/2$,  
as that of a component of $QBSO(3)_+$. 
In view of 
Proposition 2.2 below, Hatcher's result implies Theorem \ref{handlebody-thm}
for the stable range and also implies that the even MMM classes
freely generate the cohomology ring of $H_g$ in the stable range.
\end{remark}

Theorem A is a special case of the more general Theorem
\ref{mainthm} below which is a statement about the diffeomorphism groups of 
manifolds of any dimension. 
Recall that for $g\geq 2$, 
the mapping class group $\Gamma _g$ is 
homotopy equivalent to the diffeomorphism group $\Diff (\Sigma _g)$ 
\cite{Earle-Eells}, and
the handlebody subgroup 
is homotopy equivalent to the diffeomorphism group of a 3-dimensional
handlebody of genus $g$ \cite{Hatcher1,Hatcher2}. Thus the
discrete mapping class groups may be replaced by the diffeomorphism groups. 
The more general result is
about how generalizations of the MMM-classes are
pulled back in cohomology under the restriction-to-the-boundary map, 
\[
r: B\Diff(W) \to B\Diff(\partial W),
\]
where  $W$ is  a $(d+1)$-dimensional manifold with boundary $\partial W = M$. 

More precisely,
let $\pi: E \to B$ be an oriented fibre bundle with closed fibres $M$
of dimension $d$, and  let
$T^\pi E \to E$ denote the fibrewise tangent bundle.  The generalized
MMM classes (or universal tangential classes) 
are defined by taking a monomial $X$ in the Euler class
$e$ and the Pontrjagin classes $p_i$ of $T^\pi E$
and then forming the pushforward
\[
\widehat{X}(E) :=  \pi_! X(T^\pi E) \in H^*(B;\Z),
\]
where $\pi_!: H^*(E) \to H^{*-d}(B)$ is the Gysin map of $\pi$, also known as
the integration over the fibre map.  In
particular, one obtains universal characteristic classes
$\widehat{X} \in H^*(B\Diff(M);\Z)$ by taking $E\to B$ to be the
universal $M$-bundle over $B\Diff(M)$.  In this notation
$\kappa_i =
\widehat{e^{i+1}}$ for $M=\Sigma_g$.

These generalized MMM classes have been studied intensively.
Sadykov \cite {Sadykov} shows that for $d$ even
they are the only rational 
characteristic classes of $d$-dimensional manifolds that are stable
in an appropriate sense. Ebert \cite{Ebert} furthermore shows that
for each of these classes there is a bundle of $d$-manifolds on which
it does not vanish (though this is not quite the case when $d$ 
is odd \cite{Ebert2}).

\begin{thmA}\label{mainthm}
  Suppose $W$ is an oriented manifold with boundary.  Then
  $r^*\widehat{X} \in H^*(B\Diff(W);\Z)$ vanishes whenever the dimension
  of $W$ is even, or  whenever it is odd and $X$ can be
  written as a monomial just in the Pontrjagin classes.
\end{thmA}

It is worth stating an immediate corollary of the above theorem.
\begin{corA}
  Given an oriented bundle $E\to B$ of closed manifolds, the classes
  $\widehat{X}(E)$ coming from monomials $X$ in Pontrjagin classes give
  obstructions to fibrewise oriented null-bordism of the bundle.
\end{corA}

An analogue of Theorem B for not necessarily orientable manifolds
states that in cohomology with $\Z/2\Z$ coefficients $r^*\widehat X$
is trivial for any monomial $X$ in the Stiefel-Whitney classes.

We shall take a geometric approach to the mapping class groups
that was first introduced   in
\cite{Madsen-Tillmann}. From this point of view
the universal MMM-classes can be
interpreted as elements in the (stable) cohomology of the infinite loop space
associated to a certain Thom spectrum denoted by $\MT(2)$.  
More generally, 
the proof of Theorem \ref{mainthm} comes out of the theory of the Thom
spectra $\MT(d)$ (defined below in section \ref{MT-section}) and is
related to the theory of ``spaces of manifolds'' or cobordism
categories as in \cite{GMTW, Genauer}, although we do
not actually rely on their results. 

Recall, there is a homotopy
fibre sequence of infinite loop spaces
\begin{equation}\label{fundamental-sequence}
\Omega^\infty \MT(d+1) \to \QQ BSO(d+1)_+ \stackrel{\delta}{\to}
\Omega^\infty \MT(d).
\end{equation}
A bundle of oriented $d$-manifolds over a base $B$ has a classifying
map $B \to \Omega^\infty \MT(d)$, and the generalized MMM-classes are
pulled back from universal classes in the cohomology of this infinite
loop space.  A simple calculation in section 2.3 shows
that $\delta^*\widehat{X} = 0$ if
and only if either $d$ is odd or $d$ is even and $X$ can be written as
a product of Pontrjagin classes (i.e. using only even powers of the
Euler class).  The proof of Theorem \ref{mainthm} consists of
observing, see section 3,  
that if a bundle of $d$-manifolds is the fibrewise boundary
of a bundle of $(d+1)$-manifolds with boundary then its classifying
map factors up-to-homotopy through $\QQ BSO(d+1)_+$.   This
factorization trick is motivated by the philosophy that the homotopy
fibre sequence \eqref{fundamental-sequence} corresponds to the exact
sequence
\[
\{\mbox{closed $(d+1)$-manifolds}\} \hookrightarrow
\{\mbox{$(d+1)$-manifolds with boundary}\} \stackrel{\partial}{\to}
\{\mbox{closed $d$-manifolds}\}.
\]

\subsection*{Acknowledgements}
The first author thanks Oscar Randal-Williams for helpful discussions.

\section{A cofibre sequence of Thom spectra}\label{MT-section}

For the reader's convenience we will recall the definition and
construction of the fibre sequence \eqref{fundamental-sequence} and
compute the map $\delta$ in cohomology.

\subsection{Definition of the spectra}
Let $\gamma_d$ denote the tautological bundle of oriented $d$-planes
over $BSO(d)$, and let $\MT(d)$ denote the Thom spectrum,
$\bTh(-\gamma_d)$, of the virtual bundle $-\gamma_d$.  Explicitly, let
$G_{d,n}$ denote the Grassmannian of oriented $d$-planes in
$\R^{d+n}$, let $\gamma_{d,n}$ denote the tautological $d$-plane
bundle over it, and let $\gamma_{d,n}^{\perp}$ denote the
complementary $n$-plane bundle.  The $(d+n)^{th}$ space of the
spectrum $\MT(d)$ is the Thom space
\[
\Th(\gamma_{d,n}^{\perp}).
\]  
The space $G_{d,n}$ sits inside $G_{d,n+1}$ and the restriction of
$\gamma_{d,n+1}^\perp$ to $G_{d,n}$ is canonically $\gamma_{d,n}^\perp
\oplus \R$.  The structure maps of the spectrum are defined by the
composition
\[
\Sigma \Th(\gamma_{d,n}^\perp) \simeq \Th(\gamma_{d,n}^\perp \oplus \R) \simeq
\Th(\gamma_{d,n+1}^\perp|_{G_{d,n}}) \hookrightarrow
\Th(\gamma_{d,n+1}^\perp).
\]

\subsection{A homotopy cofibre sequence of Thom spectra}
The suspension spectrum $\Sigma^\infty BSO(d+1)_+$ can be regarded as
the Thom spectrum of the trivial bundle of rank $0$.  In explicit
terms, the $(d+1 +n)^{th}$ space is $\Th(\gamma_{d+1,n}\oplus \gamma_{d+1,n}^
\perp)$
and the structure maps are as above.  The inclusion 
\begin{equation}\label{thom-inclusion}
\Th(\gamma_{d+1,n}^\perp) \hookrightarrow \Th(\gamma_{d+1,n}\oplus 
\gamma_{d+1,n}^\perp)
\end{equation}
induces a map of spectra
\begin{equation}\label{spectrum-map-1}
\MT(d+1) \to \Sigma^\infty BSO(d+1)_+.
\end{equation}
The cofibre of \eqref{spectrum-map-1} is known to be homotopy equivalent to
$\MT(d)$; for convenience we include a proof here.  

\begin{lemma}\label{cofibre-lemma}
Let $E$ and $F$ be vector bundles over a base $B$, let $p: S(F) \to B$
be the unit sphere bundle of $F$, and let $L$ denote the tautological
line bundle on $S(F)$.  There is a cofibre sequence
\[
Th(E) \hookrightarrow Th(E\oplus F) \stackrel{\delta}{\to} Th(p^*E \oplus L).
\]
\end{lemma}
\begin{proof}
Observe that the quotient space $Th(E\oplus F) / Th(E)$
consists of a basepoint together with the space of all triples $(b\in
B, u \in E_b, v \in F_b \smallsetminus \{0\})$, suitably topologised.
Sending
\[
(b,u,v) \mapsto \left(\frac{v}{|v|} \in S(F_b), \:\: u \in (p^*E)_{v/|v|},
\:\: \mathrm{log}(|v|) \cdot \frac{v}{|v|} \in
L_{v/|v|}\right)
\] 
defines a homeomorphism $Th(E\oplus F) / Th(E) \cong Th(p^*E \oplus
L)$.
\end{proof}

Under the identification of the above lemma, one can see that $\delta$
corresponds to the map defined by collapsing the complement of an
appropriate tubular neighbourhood of the embedding $j: S(F)
\hookrightarrow E\oplus F$ and using the canonical identification of
the normal bundle of $j$ with $p^*E \oplus L$.  In particular, if
$E\oplus F$ is isomorphic to a trivial bundle $\mathbb{R}^n$ then
$\delta$ is the pre-transfer for the projection $p$, and hence the
Gysin map $p_!$ on cohomology is given by $\delta^*$ composed with the
Thom isomorphism.

We are concerned with the case when $B$ is the Grassmannian
$G_{d+1,n}$ of oriented $(d+1)$-planes in $\mathbb{R}^{d+1+n}$, $F$ is
the tautological $(d+1)$-plane bundle $\gamma_{d+1,n}$, and $E$ is the
complementary $n$-plane bundle $\gamma_{d+1,n}^\perp$.  In this case
there is a map
\[
q: S(\gamma_{d+1,n}) \to G_{d,n+1}
\]
given by sending $(M \in G_{d+1,n}, v\in S(M))$ to the $d$-plane
$M\cap v^\perp$.  This map is a fibration, and the fibre over a
$d$-plane $N$ is the $n$-sphere $S(N^\perp)$.  Hence $q$ is
$n$-connected.  Observe that $q^*\gamma_{d,n+1}^\perp$ is canonically
isomorphic to $p^* \gamma_{d+1,n}^\perp \oplus L$, where $p:
S(\gamma_{d+1,n}) \to G_{d+1,n}$ is the projection and $L$ is the
tautological line bundle over $S(\gamma_{d+1,n})$.  Hence there is a
map of Thom spaces,
\[
Th(p^* \gamma_{d+1,n}^\perp \oplus L) \to Th(\gamma_{d,n+1}^\perp)
\]
that is $(2n+1)$-connected.  Combining this with the above cofibre
sequence and passing to spectra indexed by $n$ now gives the desired
homotopy cofibre sequence of spectra,
\begin{equation}\label{cofib-sequence}
\MT(d+1) \to \Sigma^\infty BSO(d+1)_+  \stackrel{\widetilde{\delta}}{\to} 
\MT(d)
\end{equation}
and hence a homotopy fibre sequence of infinite loop spaces
\[
\Omega^\infty \MT(d+1) \to \QQ BSO(d+1)_+ \stackrel{\delta}{\to} 
\Omega^\infty \MT(d).
\]

\subsection{Cohomology of Thom spectra and universal tangential classes}
For any spectrum $E$ there is a map
\[
\sigma^*: H^*(E) \to \widetilde{H}^*(\Omega_0^\infty E)
\]
from the spectrum cohomology of $E$ to the reduced cohomology of the basepoint
component $\Omega ^\infty _0 E$ of the
associated infinite loop space.  This map is induced by the
evaluation map
\[
\sigma: \Sigma ^n \Omega ^n E_n \to E_n
\]
that takes $(t, f)$ to $f(t)$ for $t \in S^n$ and $f : S^n \to E_n$. Thus
$\sigma$
commutes with maps of spectra.

Let $V$ be a virtual vector
bundle of virtual dimension $-d$ over a space $B$.  There is a Thom
class, $u$, in the degree $-d$ cohomology of the associated Thom
spectrum $\bTh(V)$ (with arbitrary coefficients if $V$ is orientable
and with $\Z/2\Z$ coefficients otherwise) and by the Thom
isomorphism, the spectrum cohomology $H^*(\bTh(V))$ 
is a free  $H^*(B)$-module of rank one
generated by the Thom class $u$.  
For the Thom spectrum
$\MT(d) = \Th(-\gamma_d)$ we thus have
\[
H^*(\MT(d);\Z) \cong u\cdot H^*(BSO(d);\Z),
\]
with $\mathrm{deg}\, u = -d$. 

Now, let $X$ 
be a monomial in the Euler class $e$ and the 
Pontrjagin classes $p_i$. We
define the associated \emph {universal tangential  class} as
\[
\widehat{X} := \sigma^*(uX) \in \widetilde{H}^*(\Omega^\infty_0 \MT(d);\Z).
\]
Note that by definition all universal tangential classes are stable in the 
sense that they come from spectrum cohomology.
%These classes can be regarded as  generalized MMM
%classes as $\kappa_i = \widehat{X}$ for $X = e ^{i+1}$ and $d=2$.  
Rationally, these classes (as $X$ ranges over a basis for
the degree $> d$ monomials) freely generate the cohomology ring of
$\Omega^\infty_0\MT(d)$. 
%We will identify the cohomology of the
%basepoint component with the cohomology of any other component.

\begin{proposition}\label{computation-lemma}
Let $r = \lfloor d/2 \rfloor$ and  
$X= p_1 ^{k_1} \dots p_{r}  ^{k_r} \, e^s \in H^* (BSO(d); \Z).
$  
Consider the image of $\widehat X$
under $\delta^*: H^* (\Omega ^\infty \MT(d); \Z) \to H^*(\QQ BSO(d+1)_+;\Z)$.
%\begin{roster}
\item {(i.)}  For $d$ odd, $\delta^*\widehat{X}=0$;
\item {(ii.)}  For $d$ even, $\delta^*\widehat{X} = 0$ when $s$ is even, 
and $\delta^*\widehat{X} = 2\sigma ^*(X/e) $ when $s$ is odd.
%\end{roster}
\end{proposition}
\begin{proof}
  Identify the inclusion $BSO(d) \hookrightarrow BSO(d+1)$ with the
  projection 
  \[
  \pi: S(\gamma_{d+1}) \to BSO(d+1)
  \]
  of the unit sphere bundle of $\gamma_{d+1}$.  The map
  $\widetilde{\delta}^*: H^* (\MT(d); \Z) \to H^*(BSO(d+1);\Z)$ can
  then be identified, via the Thom isomorphism, with the Gysin map
  $\pi_!$. The image of the Euler class $e$ under the Gysin map is the
  Euler characteristic of the fiber. Thus $\pi_! e = 0$ when $d$ is
  odd and $\pi_! e = 2$ when $d$ is even. The Pontrjagin classes on
  $BSO(d)=S(\gamma_{d+1})$ are the pullbacks of the Pontrjagin classes
  on $BSO(d+1)$.  The statement now follows from the formula
  $\pi_!(\pi^* \alpha \cdot \beta) = \alpha \cdot \pi_!
  \beta$. Indeed, as $e^2=0$ for $d$ odd and $e^2 = p_{d/2}$ for $d$
  even, we may assume that $s=0$ or $s=1$ in the definition of $X$,
  and compute
\[ 
\delta^* \widehat X = \delta ^* \sigma ^* (uX) = \sigma ^* \widetilde 
\delta ^* (uX)
= \sigma ^* (\pi_! X )= \sigma ^* (p_1^{k_1} \dots p_{r} ^
{k_{r}} \pi _! \, e^s),
\]
which gives the desired result.
\end{proof}

%Applying this to the case when $d=2$  and $X = e^{i+1} $ we see that
%the MMM classes $\kappa_i$ are mapped to zero
%for $i$ odd. 

To illustrate the above result consider the case when $d=2$.
In that case we have
\[
H^* (\Omega ^\infty _0 \MT (2); \Q) = \Q[ \kappa_1, \kappa_2, \dots]
\]
with $\kappa_i = \widehat {e^{i+1}} $  of degree  $2i$,
while
\[
H^* (\QQ BSO (3); \Q ) = \Q [ \rho_1, \rho_2, \dots ]
\]
with $ \rho _i = \sigma^* {p_1 ^i} $  of degree  $ 4i$ .
Then $\delta ^* \kappa _{2i+1} = 0$ while $\delta ^* \kappa _{2i} = 2 \rho _i$.

\noindent
{\bf Remark 2.2:} 
When working over $\Z/2\Z$ (in the orientable as well as non-orientable case) 
a similar computation yields that for any 
monomial  $X$ in the Stiefel-Whitney classes $\delta^*$ maps $\widehat X$
to zero.   

\section{Classifying maps}

We show here that $\delta $ is the universal restriction-to-the-boundary 
map $r: B\Diff (W) \to B\Diff (\partial W)$.

\subsection{Bundles of closed manifolds}
Pontrjagin-Thom theory
allows one to show that the infinite loop space $\Omega^\infty \MT(d)$
classifies concordance classes of oriented $d$-dimensional
\emph{formal bundles}, which are objects slightly more general than
fibre bundles of closed oriented $d$-manifolds. Such an object over a
smooth base $B$ consists of a smooth proper map $\pi: E \to
B$ of codimension $-d$ and a bundle epimorphism $\delta \pi: TE \to
TB$ (which need not be the differential of $\pi$) with an orientation
of $ker(\delta \pi)$, cf. \cite {Madsen-Weiss}, \cite {Eliashberg-Galatius}.

For a bundle $\pi: E \to B$, the classifying map
\[
\alpha_\pi: B \to \Omega^\infty \MT(d)
\]
is defined concisely as follows.  Let $T^\pi E$ denote the fibrewise
tangent bundle.  The classifying map is the pre-transfer,
\[
\mbox{pre-trf}: B \to \Omega^\infty \bTh(-T^\pi E)
\]
followed by the map $\Omega^\infty \bTh(-T^\pi E) \to \Omega^\infty
\bTh(-\gamma_d) = \Omega^\infty \MT(d)$ induced by the classifying map
for $T^\pi E$.  To construct this map $\alpha _\pi$ 
explicitly, choose a lift of
$\pi$ to an embedding
\[
\widetilde{\pi}: E \hookrightarrow B\times \R^{d+n}
\]
for some sufficiently large $n$, and choose a fibrewise tubular
neighborhood $U \subset B\times \R^{d+n}$.  Let
$N^{\widetilde{\pi}}E$ denote the fibrewise normal bundle of
$\widetilde{\pi}$.  We obtain a map
\begin{equation}\label{collapse-map}
\Sigma^{d+n} B_+ \to \Th(N^{\widetilde{\pi}}E)
\end{equation}
by identifying $U$ with the normal bundle and collapsing the complement of
$U$ to the basepoint.  Classifying the fibrewise normal bundle gives a map
\begin{equation}\label{classify-map}
\Th(N^{\widetilde{\pi}}E) \to \Th(\gamma_{d+n}^\perp).
\end{equation}
The adjoint of the composition of \eqref{collapse-map} and \eqref{classify-map}
is a map $B \to \Omega^{d+n} \Th(\gamma_{d,n}^\perp)$ which gives the
classifying map  $\alpha_\pi$  upon composing with the map to $\Omega^\infty \MT(d)$.
One can check that the homotopy class of this map is independent of the
choice of embedding and tubular neighborhood.

The following propositions follow immediately from the description of
the classifying map in terms of the pre-transfer, and are well-known.
\begin{proposition}
The classifying map $\alpha_\pi$ is natural (up to homotopy) with respect to pullbacks.
\end{proposition}

\begin{proposition}
Given a bundle $\pi: E \to B$ and a class $\widehat{X} \in
H^*(\Omega^\infty_0\MT(d);\Z)$ defined by a monomial $X$ in the Pontrjagin
classes and Euler class on $BSO(d)$,
\[
\alpha_\pi^*\widehat{X} = \pi_!X(T^\pi E).
\]
\end{proposition}

\subsection{Bundles of manifolds with boundary}
Given a bundle, $\pi: E \to B$, of oriented $(d+1)$-manifolds with boundary,
let 
\[
\beta_\pi: B \to \QQ BSO(d+1)_+ 
\]
denote the composition of the transfer, $\mathrm{trf}: B \to \QQ E_+$,
followed by the map $\QQ E_+ \to \QQ BSO(d+1)_+$ induced by classifying
$T^\pi E$.  To construct this map $\beta _\pi$ concretely, choose an embedding
$\widetilde{\pi}$ of $E$ into $E\times \R^{d+1+n}$ over $\pi$.  A
tubular neighborhood $U$ of $E$ can then be identified with
the subspace of $T^\pi E \oplus N^{\widetilde{\pi}}E \cong E \times
\R^{d+1+n}$ consisting of those vectors for which the tangential
component is zero if they sit over the interior of a fibre and over
the boundary the tangential component is outward pointing normal to
the boundary.

Hence collapsing the complement of $U$ and classifying the fibrewise
tangent bundle gives maps
\[
\Sigma^{d+1+n} B_+ \to U^+ \to \Th(T^\pi E \oplus N^{\widetilde{\pi}}E)
\to \Th(\gamma_{d+1,n}\oplus \gamma_{d+1,n}^\perp),
\]
where $()^+$ denotes the one-point compactification.  Taking the
adjoint of this composition and then mapping into the colimit as $n$
goes to infinity gives the desired map
\[
\beta_\pi: B \to \QQ BSO(d+1)_+.
\]
Again, one can check that the homotopy class of this map is
independent of the choices made in the construction.
Analogous to $\alpha_\pi$, $\beta_\pi$ can be interpreted as the classifying map
of formal bundles of $(d+1)$-dimensional manifolds with boundary.

\begin{proposition}
  The classifying map $\beta_\pi$ is natural (up to homotopy) with
  respect to pullbacks.
\end{proposition}

%One can define a notion of an 
%oriented $(d+1)$-dimensional formal bundle with boundary and
%show that $\QQ BSO(d+1)_+$ classifies concordance classes of these.

\subsection{Restricting to the boundary}
The classifying maps $\alpha$ and $\beta$ constructed above for bundles of
closed manifolds and manifolds with boundary are compatible in two
ways.  First, it is easy to see that regarding a bundle of closed
manifolds as a bundle of manifolds with (empty) boundary is compatible with
the map $\Omega^\infty \MT(d+1) \to \QQ BSO(d+1)_+$.  More importantly for us,
the fibrewise boundary of a bundle of $(d+1)$-manifolds is a bundle of
$d$-manifolds and the classifying maps for these two bundles are
compatible in the following sense.

\begin{proposition}\label{commuting-diagram}
  Given a bundle of oriented $(d+1)$-manifolds $\pi: E \to B$ with
  fibrewise boundary bundle $\partial \pi: \partial E \to B$, the
  diagram
  \[
  \begin{diagram}
    \node{B} \arrow{s,l}{\beta_{\pi}} \arrow{se,t}{\alpha_{\partial \pi}} \\
    \node{\QQ BSO(d+1)_+} \arrow{e,b}{\delta} \node{\Omega^\infty
      \MT(d)}
  \end{diagram}
  \]
  commutes up to homotopy.
\end{proposition}
\begin{proof}
 Fix an embedding $\widetilde{\pi}: E \hookrightarrow B\times \R^{d+n}$ over
  $\pi$ and a tubular neighborhood $U$.  Let $U_{\partial}
  \subset U$ be the subspace sitting over the fibrewise boundary of
  $E$.  The lower composition in the diagram comes from the adjoint of
  a map
  \begin{equation}\label{finite-stage}
    \Sigma^{d+1+n} B_+ \to \Th(\gamma_{d+1,n}\oplus \gamma_{d+1,n}^\perp) \to
    \Th(\gamma_{d,n+1}^\perp)
  \end{equation}
  which collapses the complement of $U_{\partial}$ to the
  basepoint. The space $U_\partial$ is identified with the subspace of
  the vector bundle $(T^\pi E \oplus N^{\widetilde{\pi}}E)
  |_{\partial E}$ consisting of vectors for which the tangential
  component is outward pointing normal to $\partial E$.  In the
  fibre over any point $p \in \partial E$ there is a unique point
  $v_p$ which is sent by the map \eqref{finite-stage} to the zero
  section in $\Th(\gamma_{d,n+1}^\perp)$.  Explicitly, the component of
  $v_p$ that is normal to $E$ is zero and the tangential component is
  outward pointing normal to $\partial E$ and has length equal to
  $\varphi^{-1}(0)$, where $\varphi: (0,\infty) \to \R$ is the
  diffeomorphism used in defining the map $\delta: \QQ BSO(d+1)_+ \to
  \Omega^\infty \MT(d)$.  The map
  \[
  \rho: p \mapsto v_p \in U_\partial \subset B \times \R^{d+1+n}
  \]
  gives an embedding of $\partial E$ into $B\times \R^{d+1+n}$ over
  $\partial\pi$.  Observe that $U_\partial$ is a tubular
  neighborhood of the embedding $\rho$, and the composition
  \eqref{finite-stage} collapses the complement of $U_\partial$,
  identifies it with the normal bundle of $\rho$ and classifies this
  bundle of oriented $(n+1)$-planes in $\R^{d+1+n}$.  Hence the lower
  composition in the diagram in the statement of the proposition is a map
  $\alpha'_{\partial \pi}$ constructed exactly as $\alpha_{\partial
    \pi}$ but with a different choice of embedding and tubular
  neighborhood.  Since different choices lead to homotopic maps, the
  diagram commutes up to a homotopy.
\end{proof}

\section{Proofs of the theorems}

Theorem \ref{mainthm} follows directly from Proposition
\ref{commuting-diagram} and Proposition \ref{computation-lemma}.  Theorem
\ref{handlebody-thm} is the special case when $W$ is a 3-dimensional
oriented handlebody of genus $g \geq 2$.

\end{document}